\documentclass[12pt]{amsart}

\usepackage{amssymb,amsmath,amsthm}
\usepackage{hyperref}

\newtheorem{theorem}{Theorem}[section]
\newtheorem{lemma}[theorem]{Lemma}
\newtheorem{prop}[theorem]{Proposition}

\newtheorem{defn}{Definition}[section]

\def \R{\mathbb{R}}
\def \mR{\mathbb{R}}
\def \C{\mathbb{C}}

\def \e{\varepsilon}
\def \eps{\varepsilon}

\def \Om{\Omega}

\def \Lap{\triangle}
\def \grad{\nabla}

\def \half{1/2}

\newcommand{\abs}[1]{\lvert #1 \rvert}          
\newcommand{\norm}[1]{\lVert #1 \rVert}         
\newcommand{\supp}{\mathrm{supp}}

\renewcommand{\phi}{\varphi}

\theoremstyle{definition}
\newtheorem{thm}{Theorem}[section]

\newtheorem*{remark}{Remark}
\newtheorem*{remarks}{Remarks}

\numberwithin{equation}{section}

\newcounter{sidenote}
\begin{document}

\title[Partial data inverse problems for Maxwell]{Partial data inverse problems for Maxwell equations via Carleman estimates}

\author[Chung]{Francis J. Chung}
\address{Department of Mathematics \& Statistics, University of Michigan, Ann Arbor, USA}

\author[Ola]{Petri Ola}
\address{Department of Mathematics \& Statistics, University of Helsinki, Helsinki, Finland}

\author[Salo]{Mikko Salo}
\address{Department of Mathematics \& Statistics, University of Jyv\"{a}skyl\"{a}, Jyv\"{a}skyl\"{a}, Finland}

\author[Tzou]{Leo Tzou}
\address{School of Mathematics and Statistics, Sydney University, Sydney, Australia}



\begin{abstract}
In this article we consider an inverse boundary value problem for the time-harmonic Maxwell equations. We show that the electromagnetic material parameters are determined by boundary measurements where part of the boundary data is measured on a possibly very small set. This is an extension of earlier scalar results of Bukhgeim-Uhlmann and Kenig-Sj\"ostrand-Uhlmann to the Maxwell system. The main contribution is to show that the Carleman estimate approach to scalar partial data inverse problems introduced in those works can be carried over to the Maxwell system.
\end{abstract}

\maketitle


\section{Introduction}

In this paper we discuss an inverse problem for the time-harmonic Maxwell equations with partial data, and show that the electromagnetic material parameters are determined by measurements on certain parts of the boundary. The result is new even for the case of bounded domains in $\mR^3$, but it will be convenient to formulate it more generally on compact manifolds with boundary.

Let $(M,g)$ be a compact oriented Riemannian $3$-manifold with $C^{\infty}$ boundary. The electric and magnetic fields on $M$ are described, respectively, by $1$-forms $E$ and $H$ which satisfy the Maxwell equations in $M$:
\begin{equation}
\label{maxwell equations}
\left\{ \begin{array}{rl}
*dE &\!\!\!= i\omega \mu H, \\
*dH &\!\!\!= -i\omega \eps E.
\end{array} \right.
\end{equation}
Here $\omega > 0$ is a fixed frequency, $d$ is the exterior derivative, and $*$ is the Hodge star operator for the metric $g$. The material parameters $\varepsilon$ and $\mu$ are assumed to be complex valued functions in $C^3(M)$ and to satisfy
\[\mathrm{Re}(\varepsilon)>0,\ \mathrm{Re}(\mu)>0. \]

The inverse problem is formulated in terms of partial measurements of the boundary tangential traces $tE$ and $tH$ of $E$ and $H$. Here, the tangential trace of a $k$-form $\eta$ is defined by 
\begin{equation*}
t: \eta \mapsto i^* \eta,
\end{equation*}
where $i: \partial M \to M$ is the inclusion map. If $\Gamma_1$ and $\Gamma_2$ are open subsets of $\partial M$, we define the partial Cauchy data set
\begin{equation*}
C^{\Gamma_1,\Gamma_2} = \{ (tE|_{\Gamma_1}, tH|_{\Gamma_2}) \,;\, (E,H) \in H^2(M, \Lambda^1 M)^2 \text{ solves \eqref{maxwell equations} and } \supp(tE) \subset \Gamma_1 \}.
\end{equation*}
As described in the end of Section \ref{sec:sec_reductions}, if $\omega$ is outside a discrete set of resonant frequencies, then the knowledge of $C^{\Gamma_1, \Gamma_2}$ is equivalent with knowing the partial admittance map 
$$
\Lambda^{\Gamma_1, \Gamma_2}:  f \mapsto tH|_{\Gamma_2},
$$
where $f$ is a $1$-form on $\partial M$ with $\supp(f) \subset \Gamma_1$, and $(E,H)$ is the unique solution of \eqref{maxwell equations} satisfying $tE = f$ on $\partial M$. This corresponds to prescribing $tE$ on $\Gamma_1$ and measuring $tH$ on $\Gamma_2$. The inverse problem is to determine the coefficients $\eps$ and $\mu$ from the knowledge of $C^{\Gamma_1,\Gamma_2}$ for some choices of $\Gamma_j$.

One may think of the above inverse problem for Maxwell equations as a generalization to systems of the inverse conductivity problem introduced by Calder\'on \cite{calderon}. In the case where $M$ is a bounded domain in $\mR^3$ and $g$ is the Euclidean metric, it was proved in \cite{OPS} that measurements on the full boundary ($\Gamma_1 = \Gamma_2 = \partial M$) determine $\varepsilon$ and $\mu$ uniquely. Earlier results include \cite{SIC, SuU, CP}, and a simplified proof was presented in \cite{OS}. Stability results for this inverse problem are in \cite{Caro_stability1, Caro_stability2}, boundary determination results are in \cite{McD1, JM}, the case of chiral media is considered in \cite{McD2}, and a recent uniqueness result for $C^1$ coefficients is given in \cite{CZ}.

The inverse problem for Maxwell equations on manifolds was discussed in \cite{OPS_insideout}. The first uniqueness results in the non-Euclidean case were given in \cite{KSaU}, again for the full data case $\Gamma_1 = \Gamma_2 = \partial M$, when $M$ is an admissible manifold in the following sense.

\begin{defn}
\label{def of admissible}
A compact Riemannian manifold $(M,g)$ with smooth boundary is called \emph{admissible} if $(M,g)$ is embedded in $(T,g)$ where $T = \R \times M_0$ and $g = c (e\oplus g_0)$, where $c$ is a smooth positive function, $(\R,e)$ is the Euclidean line, and $(M_0,g_0)$ is a \emph{simple} manifold (a compact manifold with smooth strictly convex boundary such that the exponential map $\exp_p$ is a diffeomorphism onto $M_0$ for each $p \in M_0$).
\end{defn}

Locally, admissibility implies that in some local coordinates $x = (x_1,x')$ the metric takes the form 
$$
g(x_1,x') = c(x_1,x') \left( \begin{array}{cc} 1 & 0 \\ 0 & g_0(x') \end{array} \right)
$$
where $g_0(x')$ is some positive definite $(n-1)\times(n-1)$ matrix. Conversely, any metric that takes this form in some local coordinates is admissible provided that $g_0$ is simple. Admissible manifolds include compact submanifolds of Euclidean space, hyperbolic space, and $S^3$ minus a point. They also include sufficiently small submanifolds of conformally flat manifolds and warped products. For more information see \cite{DKSaU, afgr}.

In this article we will improve the results of \cite{OPS, KSaU} by considering the partial data problem where $\Gamma_1 = \partial M$ but $\Gamma_2$ is an appropriately chosen open subset of $\partial M$ which is defined by a limiting Carleman weight (LCW) of the manifold $M$. The notion of LCWs was introduced in connection with the Calder\'on problem with partial data in \cite{KSU} in Euclidean space. LCWs were analysed in detail in \cite{DKSaU} also on manifolds. In particular, the class of admissible manifolds emerged in \cite{DKSaU} as a natural class where LCWs exist and one may expect to be able to solve related inverse problems. Note that any admissible manifold has global coordinates $(x_1,x')$ where $x_1$ is the Euclidean direction. It was shown in \cite{DKSaU} that $\phi(x) = x_1$ is a natural LCW in this setting, and in this paper we will always assume that the LCW is given by $\phi(x) = x_1$.


We define the ``front face" of the boundary with respect to the LCW $\phi$ by
\[F_\phi := \{p\in \partial M \mid \langle d\phi(p), \nu\rangle \geq 0\}.\]
Here $\nu$ is the $1$-form corresponding to the unit outer normal of $\partial \Omega$. Our main result is the following theorem.

\begin{thm}
\label{main theorem}
Let $(M,g)$ be an admissible 3-manifold and $\phi(x) = x_1$ a LCW on $M$. Let $\eps_j, \mu_j \in C^3(M)$ be complex functions with positive real part, and let $\omega > 0$. Suppose that the corresponding Cauchy data sets satisfy 
\begin{equation*}
C_1^{\partial M,\tilde{F}} = C_2^{\partial M,\tilde{F}}
\end{equation*}
for some open neighborhood $\tilde{F}$ of $F_\phi$  in $\partial M$. Assume in addition that $\eps_1 = \eps_2$ and $\mu_1 = \mu_2$ to second order on $\partial M$. Then $\eps_1 \equiv \eps_2$ and $\mu_1 \equiv \mu_2$ in $M$.
\end{thm}

We mention a particular case of the above theorem in $\R^3$ ($\text{ch}(\overline{\Omega})$ denotes convex hull).

\begin{thm}
\label{main theorem2}
Let $\Omega \subset \mR^3$ be a bounded domain with $C^{\infty}$ boundary, let $\eps_j, \mu_j \in C^3(\overline{\Omega})$ be complex functions with positive real part, and let $\omega > 0$. Given $x_0 \in \mR^3 \setminus \text{ch}(\overline{\Omega})$, define 
\begin{equation*}
F(x_0) = \{ x \in \partial \Omega \,;\, (x-x_0) \cdot \nu(x) \leq 0 \}.
\end{equation*}
Suppose that the corresponding Cauchy data sets satisfy 
\begin{equation*}
C_1^{\partial \Omega,\tilde{F}} = C_2^{\partial \Omega,\tilde{F}}
\end{equation*}
for some open neighborhood $\tilde{F}$ of $F(x_0)$  in $\partial \Omega$. Assume in addition that $\eps_1 = \eps_2$ and $\mu_1 = \mu_2$ to second order on $\partial \Omega$. Then $\eps_1 \equiv \eps_2$ and $\mu_1 \equiv \mu_2$ in $\Omega$.
\end{thm}

This result involves the logarithmic Carleman weight in $\mR^3$ and corresponds to one of the results of \cite{KSU} for the partial data problem for the Schr\"odinger equation. Note that if $\Omega$ is strictly convex, then any open subset of $\partial \Omega$ can serve as $\tilde{F}$ above.

A standard method to study such partial data problems is to apply the idea developed in \cite{BU, KSU} for scalar equations, where Carleman estimates with suitable positive boundary terms were used to suppress the unknown information. The partial data inverse problem for Maxwell equations presents several challenges when one tries to apply these methods directly. The first difference is the fact that the principal part of this equation is a first order system rather than a second order scalar equation, and Carleman estimates with boundary terms for systems seem to be more involved than those for scalar equations (see \cite{Eller, ST} for some such estimates). Secondly, the fact that ellipticity is somewhat ''hidden'' in the system makes it difficult to obtain a estimate similar to the one used in \cite{KSU} which could be directly applied to the Maxwell system.

We circumvent these difficulties by first using ideas from \cite{OS, OPS_insideout, KSaU} to show that the problem reduces to constructing suitable solutions to a Hodge Dirac system 
\[(P + W) u = 0,\ \ \  tu|_{\partial M \setminus \tilde F} = 0.\]
We will then apply the recent work \cite{CST} where the authors studied Carleman estimates and complex geometrical optics solutions for the Hodge Laplace operator with relative and absolute boundary conditions. The crucial point is that the Carleman estimate with boundary terms proved in \cite{CST} for the Hodge Laplacian is sufficiently powerful to yield information also in the Maxwell case. However, we stress that the present situation for Maxwell equations does not reduce to that of \cite{CST}, since the boundary measurements for Maxwell determine in a sense only part of the boundary measurements for the Hodge system. It is essential to use the special structure of the Maxwell system to show that the coefficients are uniquely determined.

The main contribution of this work is to show that the Carleman estimate approach to scalar partial data inverse problems introduced in \cite{BU, KSU} can be carried over to the Maxwell system. A Carleman estimate approach for a different first order system, related to the Pauli Dirac operator, was presented in  \cite{ST} where the method involved decoupling the Pauli Dirac operator into a second order differential operator with the Laplacian as its principal part. A suitable Carleman estimate was then applied to the decoupled equation to recover the coefficients. The partial data problem for Maxwell equations was also studied in \cite{COS} in the case when $M$ is Euclidean and the inaccessible part of the boundary is a portion of a hyperplane or a  sphere, following the scalar approach of \cite{I}. Due to the partial symmetry of the domain one can reflect across the flat part of the boundary and reduce the problem to a full data problem. Partial data results are also known for a two-dimensional Maxwell system \cite{IY_maxwell_2D} and for Maxwell equations in a waveguide \cite{IY_maxwell_waveguide}, based on the two-dimensional partial data results of \cite{IUY}, and in the case where the parameters are known near the boundary \cite{BMR}. Recently, extensions of the Carleman estimate and reflection approaches for partial data problems were introduced in \cite{KS} and \cite{IY_3D}; some of these extensions for the Maxwell system were considered in \cite{IY_maxwell_cylindrical}. We also remark that the methods in the recent paper \cite{DKLS} might allow to relax to some extent the admissibility assumption in Theorem \ref{main theorem}.

\subsection*{Acknowledgements}

F.C., P.O.\ and M.S.\ were partly supported by the Academy of Finland (Centre of Excellence in Inverse Problems Research), and F.C.\ and M.S.\ were supported by an ERC Starting Grant (grant agreement no 307023). L.T.\ was partly supported by the Academy of Finland (decision no 271929), Vetenskapsr\aa det (decision no 2012-3782), and Australian Research Council Future Fellowship (fellowship no G161397). F.C.\ would like to acknowledge the University of Jyv\"{a}skyl\"{a} for its hospitality on subsequent visits.

\section{Reduction to the Hodge Laplacian}\label{sec:sec_reductions}
In this section we employ ideas from \cite{OS, OPS_insideout, KSaU} to show that the problem reduces to constructing suitable solutions of the Hodge Dirac and Schr\"odinger operator. The notation follows closely \cite{KSaU} and we refer to that article for more details. If $E$ and $H$ are complex $1$-forms in $M$, we consider the graded form $X = \Phi + E + *H + *\Psi$ where $\Phi$ and $\Psi$ are complex scalar functions. We write graded forms in the vector notation 
\begin{equation*}
X = \left( \begin{array}{cc|cc} \Phi & *H & *\Psi & E \end{array} \right)^t.
\end{equation*}
The line denotes that the forms of even order and odd order are grouped together, which will result in a block structure for the equation. We define the following Hodge Dirac operator acting on graded forms in $M$, 
\begin{equation*}
P = \frac{1}{i} (d-\delta) = \frac{1}{i} \left( \begin{array}{cc|cc} & &  & -\delta \\ & & -\delta & d \\ \hline  & d & & \\ d & -\delta & & \end{array} \right).
\end{equation*}
The Dirac equation which is most closely related to Maxwell is 
\begin{equation*}
(P + V)X = 0
\end{equation*}
where $V$ is the operator acting on graded forms as the matrix operator 
\begin{equation*}
V = \left( \begin{array}{cc|cc} -\omega \mu & & & *D\alpha \wedge * \\ & -\omega \mu & *D\alpha \wedge * & \\ \hline & D\beta \wedge & -\omega \eps & \\ D\beta \wedge & & & -\omega \eps \end{array} \right).
\end{equation*}
and where $D = \frac{1}{i} d$, $\alpha = \log \,\eps$, and $\beta = \log\,\mu$.  One easily sees that 
\begin{equation}
\left\{ \begin{array}{rl}
(P+V) X =0,\ \ \ \ \ \ \ \ \ \ \ \ \ \   \\
X = \left( \begin{array}{cc|cc} 0 & *H & *0 & E \end{array} \right)^t
\end{array} \right.\Leftrightarrow\left\{ \begin{array}{rl}
*dE &\!\!\!= i\omega \mu H, \\
*dH &\!\!\!= -i\omega \eps E.
\end{array} \right.
\end{equation}
The advantage here is that $P$ is an elliptic operator.

For the reduction to the Schr\"odinger equation with Hodge Laplacian, we consider rescaled solutions 
\begin{equation*}
Y = \left( \begin{array}{c|c} \mu^{1/2} & \\ \hline & \eps^{1/2} \end{array} \right) X,
\end{equation*}
and henceforth relate $X$ and $Y$ in this way. We will write 
\[
Y = \left( \begin{array}{cc|cc} Y^0 & Y^2 & Y^3 & Y^1 \end{array} \right)^t
\]
where $Y^k$ is the $k$-form part of the graded form $Y$. It is easy to see that 
\begin{equation*}
(P+V)X = 0 \quad \Longleftrightarrow \quad (P+W)Y = 0
\end{equation*}
where $W$ is the potential 
\begin{equation*}
W = -\kappa + \frac{1}{2} \left( \begin{array}{cc|cc} & & & *D\alpha \wedge * \\ & & *D\alpha \wedge * & -D\alpha \wedge \\ \hline & D\beta \wedge & & \\ D\beta \wedge & *D\beta \wedge * & & \end{array} \right),
\end{equation*}
and $\kappa = \omega (\eps \mu)^{1/2}$.

The reason for using the rescaled solutions $Y$ is that there is a good reduction of the Dirac equation $(P+W)Y = 0$ to a Schr\"odinger equation with no first order terms (see \cite[Lemma 3.1]{KSaU}).

\begin{lemma} \label{lemma:reduction_schrodinger}
Denote by $-\Delta = d\delta + \delta d$ the Hodge Laplacian acting on graded forms. We have 
\begin{align*}
(P+W)(P-W^t) &= -\Delta + Q, \\
(P-W^t)(P+W) &= -\Delta + Q', \\
(P+W^*)(P-\bar{W}) &= -\Delta + \hat{Q},
\end{align*}
where $Q$, $Q'$, and $\hat{Q}$ are continuous potentials (endomorphisms of $\Lambda M$) with 
\begin{align*}
Q &= - \kappa^2 + \frac{1}{2} \left( \begin{array}{cc|cc} \Delta \alpha + \frac{1}{2} \langle d\alpha,d\alpha \rangle & 0 & \bullet & \bullet \\ 0 & \bullet & \bullet & \bullet \\ \hline \bullet & \bullet & \Delta \beta + \frac{1}{2} \langle d\beta,d\beta \rangle & 0 \\ \bullet & \bullet & 0 & \bullet \end{array} \right), \\
Q' &= - \kappa^2 - \frac{1}{2} \left( \begin{array}{cc|cc} \Delta \beta - \frac{1}{2} \langle d\beta,d\beta \rangle & 0 & 0 & 0 \\ \bullet & \bullet & \bullet & \bullet \\ \hline 0 & 0 & \Delta \alpha - \frac{1}{2} \langle d\alpha,d\alpha \rangle & 0 \\ \bullet & \bullet & \bullet & \bullet \end{array} \right),
\end{align*}
and $\bullet$ denote bounded coefficients involving derivatives of $\eps$ and $\mu$ up to second order. 
\end{lemma}

\begin{remarks}
The coefficient functions denoted simply by a ``bullet" above are not needed in our argument. They are easy to compute though if one so wishes. Also
observe that due to this lemma, if $Z$ is a solution of $(-\Delta + \hat Q) Z =0$, then $Y := (P-\bar W)Z$ would solve $(P+W^*) Y = 0$. Finally, the precise form of \(\hat{Q}\) is not needed in this work, so we will not bother the reader by writing it out explicitly.
\end{remarks}


The above arguments can be used to reduce the Maxwell equations to the Hodge Dirac and Schr\"odinger equations. We now want to give a similar reduction on the level of boundary measurements. In fact, we will prove an integral identity showing - see Lemma \ref{integral identity with boundary conditions} -  that if the partial Cauchy data for two sets of coefficients coincide, then the difference of the corresponding potentials $Q_j$ is orthogonal to a product of solutions to the Schr\"odinger and Dirac equations. This result is a version of \cite[Lemma 3.2]{KSaU} adapted to the partial data problem. 

If $U$ and $V$ are graded forms on $M$, we use the notations 
\begin{equation*}
(U|V) = \int_{M} \sum_{j=0}^3 \langle U^j, \overline{V}^j \rangle \,dV, \quad (U|V)_{\partial M} = \int_{\partial M} \sum_{j=0}^3 \langle U^j, \overline{V}^j \rangle \,dS.
\end{equation*}
for the \(L^2\)--inner products. We also use the same notation in the case when \(U\) and \(V\) are in appropriate dual Sobolev spaces.
Note that in the second expression, $U$ and $V$ do not have be graded forms of the manifold $\partial M$ (for instance, the $1$-form parts may have a normal component). We have the integration by parts formulas (\cite[Section 2.10]{T1})
\begin{align}
 & (dU|V) = (\nu \wedge U|V)_{\partial M} + (U|\delta V), \label{iparts1} \\
 & (\delta U|V) = -(i_{\nu} U|V)_{\partial M} + (U|dV) \label{iparts2}
\end{align}
where $\nu$ is the outward normal and $i_{\nu}$ is the contraction with $\nu$. Also, we define 
\[
H_\Delta = \{u\in L^2 (M, \Lambda M);\, \Delta u \in L^2 (M, \Lambda M)\}
\]
and equip this with the natural graph norm 
\[
\| u\| _{H_\Delta} ^2  = \|u\|^2 + \|\Delta u\|^2.
\]
Here and in the rest of this paper, notation of the form $\| \cdot \|$ and $\| \cdot \|_{M}$ should be taken to indicate the relevant $L^2$ norm unless otherwise specified.

One feature of the method in this paper is that we need to deal with pairs of solutions, where one solution can be very nonsmooth (only $H^{-1}$ in general). For this reason we will need the following result concerning smooth approximation and traces in the $H_{\Delta}$ space. The proof is given in the appendix.

\begin{lemma} \label{lemma_hdelta_properties}
The set $C^{\infty}(M, \Lambda M)$ is dense in $H_{\Delta}$. The trace maps $u \mapsto (tu, t\delta u)$, $u \mapsto (t*u, t\delta *u)$, $u \mapsto (tu, tdu)$ and $u \mapsto (t*u, td*u)$, initially defined on $C^{\infty}(M, \Lambda M)$, have unique bounded extensions as maps 
\[
H_{\Delta} \to H^{-1/2}(\partial M, \Lambda \partial M) \times H^{-3/2}(\partial M, \Lambda \partial M).
\]
\end{lemma}

Given a pair of electromagnetic parameters  \(\{\eps _j, \mu _j\}, \, j = 1,\, 2\), we define the corresponding matrix potentials \(V_j\), \(Q_j\) and \(W_j\) be replacing \(\eps\) and \(\mu\) respectively with \(\eps _j\) and \(\mu _j\) in definitions above. The next result is the main integral identity.

\begin{lemma}
\label{integral identity with boundary conditions}
Let $\eps_j, \mu_j \in C^{3}(M)$, \(j=1,\,2\), be complex functions with positive real parts, and assume that 
\begin{equation} \label{epsmu_boundary}
\eps_1 = \eps_2,\, \, \mu_1 = \mu_2 \,\, \text{to second order on $\partial M$}.
\end{equation}
Let $\Gamma_2$ be any open subset of $\partial M$, and assume that $C_1^{\partial M,\Gamma_2} = C_2^{\partial M,\Gamma_2}$. Then 
\begin{equation*}
((Q_1-Q_2)Z_1|Y_2) = 0
\end{equation*}
for any graded forms $Z_1 \in H^3$ and $Y_2 \in H^{-1}$ such that 
\begin{eqnarray*}
 & (-\Delta+Q_1) Z_1 = 0, & \\
 & Y_1 = (P-W_1^t)Z_1 \mbox{ satisfies } Y_1^0 = Y_1^3 = 0, & 
\end{eqnarray*}
and
\begin{eqnarray*}
 & Y_2 = (P-\bar{W}_2)Z_2 \mbox{ for some } Z_2 \in H_{\Delta} \mbox{ with } (-\Delta+\hat{Q}_2) Z_2 = 0, & \\
 & \mathrm{supp}(t Y_2) \subset \Gamma_2.  &  
\end{eqnarray*}
\end{lemma}

\begin{remark}
By the assumptions for $\eps_j$ and $\mu_j$, we have that $Q_1 - Q_2$ is $C^1$ and vanishes on $\partial M$. Thus $(Q_1-Q_2)Z_1$ is $H^1$ and vanishes on $\partial M$, and the pairing $((Q_1-Q_2)Z_1|Y_2)$ where $Y_2 \in H^{-1}$ makes sense. Note also that since $Z_2 \in H_{\Delta}$, Lemma \ref{lemma_hdelta_properties} guarantees that the boundary value $tY_2$ exists as an element of $H^{-3/2}$. Here, $Z_1$ and $Y_1$ need to be related to a solution for Maxwell, but $Y_2$ only needs to solve a Dirac equation. This simplifies the recovery of coefficients. 
\end{remark}

\begin{proof}[Proof of Lemma \ref{integral identity with boundary conditions}]
By the discussion earlier in this section, if $Y_1$ is as stated, then $Y_1 \in H^2$ solves $(P+W_1) Y_1 = 0$ and gives rise to a $H^2$ solution $X_1 = \left( \begin{array}{cc|cc} 0 & *H_1 & 0 & E_1 \end{array} \right)^t$ of the Maxwell equations $(P+V_1)X_1 = 0$. By the assumption on Cauchy data sets, there is $\tilde{X}_2 = \left( \begin{array}{cc|cc} 0 & *\tilde{H}_2 & 0 & \tilde{E}_2 \end{array} \right)^t \in H^2$ satisfying 
\begin{eqnarray*}
 & (P + V_2) \tilde{X}_2 = 0, & \\
 & t\tilde{E}_2|_{\partial M} = tE_1|_{\partial M}, \quad t\tilde{H}_2|_{\Gamma_2} = tH_1|_{\Gamma_2}. & 
\end{eqnarray*}
Let $\tilde{Y}_2 \in H^2$ be the solution of $(P+ W_2)\tilde{Y}_2 = 0$ corresponding to $\tilde{X}_2$.

We first claim that the assumption on Cauchy data sets implies 
\begin{equation} \label{w_difference_identity}
((W_1 - W_2) Y_1 | Y_2)_M = 0.
\end{equation}
Here, $Y_2$ is only in $H^{-1}$ since $Y_2 = (P-\bar{W}_2) Z_2$ where $Z_2 \in H_{\Delta}$, but the pairing is well defined because the quantity $(W_1-W_2) Y_1$ is in $H^1$ with vanishing boundary values since $\eps_1 = \eps_2$ and $\mu_1 = \mu_2$ to second order on $\partial M$. To deal with the nonsmooth solution $Y_2$, we use Lemma \ref{lemma_hdelta_properties} and choose $Z_l \in C^{\infty}(M,\Lambda M)$ with $Z_l \to Z_2$ in $H_{\Delta}$ as $l \to \infty$, which implies in particular that $(P-\bar{W}_2) Z_l \to Y_2$ in $H^{-1}$. Noting that 
\[
(P+W_2)(Y_1 - \tilde{Y}_2) = (W_2-W_1)Y_1,
\]
we compute 
\begin{align*}
 &((W_2 - W_1)Y_1 | Y_2)_M = \lim_l ((W_2 - W_1)Y_1 | (P-\bar{W}_2) Z_l)_M \\
 &= \lim_l ((P+W_2)(Y_1-\tilde{Y}_2) | (P-\bar{W}_2) Z_l)_M \\
 &= \lim_l \left[ \frac{1}{i} ( (\nu \wedge + i_{\nu})(Y_1-\tilde{Y}_2) | (P-\bar{W}_2) Z_l)_{\partial M} + (Y_1-\tilde{Y}_2 | (P+W_2^*)(P-\bar{W}_2) Z_l)_M \right].
\end{align*}
In the last line we used \eqref{iparts1} and \eqref{iparts2}. For the last term, we have 
\begin{equation} \label{zl_schrodinger_limit}
(P+W_2^*)(P-\bar{W}_2) Z_l = (-\Delta+\hat{Q}_2) Z_l \to (-\Delta+\hat{Q}_2) Z_2 = 0 \text{ in } L^2
\end{equation}
so the second term in brackets goes to zero as $l \to \infty$. For the boundary term, we claim that 
\begin{eqnarray}
t(Y_1 - \tilde{Y}_2) &=& 0 \mbox{ on } \partial M \label{by1} \\
t*(Y_1 - \tilde{Y}_2) &=& 0 \mbox{ on } \Gamma_2  \label{by2} \\
tY_2 &=& 0 \mbox{ on } \Gamma_2^c. \label{by3} 
\end{eqnarray}
Assuming these, one has $\nu \wedge (Y_1 - \tilde{Y}_2)|_{\partial M} = 0$ and $i_{\nu}(Y_1 - \tilde{Y}_2)|_{\Gamma_2} = 0$, and the boundary term becomes 
\[
\lim_l \frac{1}{i} ( Y_1-\tilde{Y}_2 | \nu \wedge (P-\bar{W}_2) Z_l)_{\Gamma_2^c}.
\]
Here $t(P-\bar{W}_2) Z_l|_{\Gamma_2^c} \to tY_2|_{\Gamma_2^c} = 0$ in $H^{-3/2}$ by Lemma \ref{lemma_hdelta_properties}. Since $Y_1-\tilde{Y}_2|_{\partial M}$ is in $H^{3/2}$, the limit of the pairing on $\Gamma_2^c$ is well defined and in fact zero. This proves \eqref{w_difference_identity} modulo \eqref{by1}--\eqref{by3}. To check \eqref{by1}, we see that 
\[
\nu \wedge (Y_1 - \tilde{Y}_2) = (\begin{array}{cc|cc} 0 & \eps_1^{\half} \nu \wedge (E_1 - \tilde{E}_2) &  *\mu_1^{\half} \langle \nu, H_1 - \tilde{H}_2 \rangle & 0 \end{array}).
\]
From the hypotheses, $\nu \wedge (E_1 - \tilde{E}_2)|_{\partial M} = 0$, and using the surface divergence as in the proof of [KSU09, Lemma 3.2], we also have $\langle \nu, H_1 - \tilde{H}_2 \rangle|_{\partial M} = 0$.  Therefore \eqref{by1} is satisfied.  Meanwhile,
\[
\nu \wedge *(Y_1 - \tilde{Y}_2) = (\begin{array}{cc|cc} 0  & \mu_1^{\half} \nu \wedge (H_1 - \tilde{H}_2) & *\eps_1^{\half} \langle \nu, E_1 - \tilde{E}_2 \rangle & 0 \end{array}).
\]
From the hypotheses, $\nu \wedge (H_1 - \tilde{H}_2)|_{\Gamma_2} = 0$, and using the surface divergence as in the proof of [KSU09, Lemma 3.2], we also have $\langle \nu, E_1 - \tilde{E}_2 \rangle|_{\Gamma_2} = 0$.  Therefore \eqref{by2} is satisfied.  The equation \eqref{by3} follows straight from the definition of $Y_2$.  

Having proved \eqref{w_difference_identity}, we next claim that 
\begin{equation} \label{w_q_difference_identity}
((W_1 - W_2) Y_1 | Y_2)_M = ( (Q_1-Q_2)Z_1 | Y_2 )_M.
\end{equation}
Note that the pairing on the right is well defined, since $(Q_1-Q_2) Z_1$ is in $H^1$ with vanishing boundary values. Since $Q_j = W_j \circ P - P \circ W_j^t - W_j W_j^t$, we have
\begin{eqnarray*}
(W_1 - W_2)Y_1 &=& (W_1-W_2)(P-W_1^t)Z_1 \\
               &=& (Q_1 + P \circ W_1^t)Z_1 - (Q_2 + P \circ W_2^t + W_2(W_2^t - W_1^t))Z_1 \\
							 &=& (Q_1 - Q_2)Z_1 + (P +W_2)W_1^t Z_1 - (P + W_2)W_2^t Z_1. \\
\end{eqnarray*}
Choosing again $Z_l \in C^{\infty}(M,\Lambda M)$ with $Z_l \to Z_2$ in $H_{\Delta}$, we have 
\begin{multline*}
((W_1-W_2)Y_1|Y_2)_M = \lim_l ((W_1-W_2)Y_1|(P-\bar{W}_2)Z_l)_M \\
= \lim_l \left[ ((Q_1 - Q_2)Z_1|(P-\bar{W}_2)Z_l)_M + ((P + W_2)(W_1^t - W_2^t) Z_1|(P-\bar{W}_2)Z_l)_M \right].
\end{multline*}
The first term in brackets has limit $((Q_1 - Q_2)Z_1|Y_2)_M$. After integrating by parts using \eqref{iparts1} and \eqref{iparts2}, the second term becomes 
\[
\lim_l \left[ \frac{1}{i} ((\nu \wedge + i_{\nu})(W_1^t - W_2^t) Z_1|(P-\bar{W}_2)Z_l)_{\partial M} + ((W_1^t - W_2^t) Z_1|(-\Delta+\hat{Q}_2)Z_l)_M \right].
\]
Since $W_1^t = W_2^t$ at $\partial M$ to first order, the boundary term vanishes. Also, by \eqref{zl_schrodinger_limit} the second term has limit zero. This proves \eqref{w_q_difference_identity}, and combining this with \eqref{w_difference_identity} proves the lemma.
\end{proof}

We now explain the connection between the Cauchy data set and the admittance map, which is a standard way for expressing boundary measurements for Maxwell equations. First we introduce some function spaces. The surface divergence of $f \in H^s(\partial M, \Lambda^1 \partial M)$ is given by 
\begin{equation*}
\text{Div}(f) = \langle d_{\partial M} f, dS \rangle
\end{equation*}
where $dS$ is the volume form on $\partial M$. Define the spaces 
\begin{gather*}
H^2_{\text{Div}}(M) = \{ u \in H^2(M, \Lambda^1 M) \,;\, \text{Div}(tu) \in H^{3/2}(\partial M) \}, \\
T H^{3/2}_{\text{Div}}(\partial M) = \{ f \in H^{3/2}(\partial M, \Lambda^1 \partial M) \,;\, \text{Div}(f) \in H^{3/2}(\partial M) \}.
\end{gather*}
By \cite[Theorem A.1 and subsequent remark]{KSaU}, there is a discrete set $\Sigma \subset \C$ such that if $\omega \in \C \setminus \Sigma$, then for any $f \in TH^{3/2}_{\mathrm{Div}}(\partial M)$ the Maxwell equations \eqref{maxwell equations} have a unique solution $(E, H) \in H^2_{\mathrm{Div}}(M) \times H^2_{\mathrm{Div}}(M)$ satisfying $tE = f$ on $\partial M$. Assume that $\omega$ is not a resonant frequency, and let $\Gamma_1$ and $\Gamma_2$ be nonempty open subsets of $\partial M$. If $TH^{3/2}_{\mathrm{Div,c}}(\Gamma_1)$ consists of those elements in $TH^{3/2}_{\mathrm{Div}}(\partial M)$ that are compactly supported in $\Gamma_1$, we may define the partial admittance map 
$$
\Lambda^{\Gamma_1, \Gamma_2}: TH^{3/2}_{\mathrm{Div, c}}(\Gamma_1) \to TH^{3/2}_{\mathrm{Div}}(\Gamma_2), \ \ f \mapsto tH|_{\Gamma_2}
$$
where $(E,H)$ is the unique solution of \eqref{maxwell equations} satisfying $tE = f$ on $\partial M$. Note also that any solution $(E,H) \in H^2 \times H^2$ of \eqref{maxwell equations} is in $H^2_{\mathrm{Div}} \times H^2_{\mathrm{Div}}$, by using the surface divergence as in \cite[Lemma 3.2]{KSaU}.

Now if $\omega$ is not a resonant frequency, it is clear that $C^{\Gamma_1,\Gamma_2}$ is the graph of $\Lambda^{\Gamma_1, \Gamma_2}$:
$$
C^{\Gamma_1,\Gamma_2} = \{Ê(f, \Lambda^{\Gamma_1, \Gamma_2} f) \,;\, f \in TH^{3/2}_{\mathrm{Div, c}}(\Gamma_1) \}.
$$
Thus $C^{\Gamma_1,\Gamma_2}$ and $\Lambda^{\Gamma_1, \Gamma_2}$ are equivalent information in this case.

\section{Carleman estimates and CGO solutions} \label{section_cgo}

Let $(M,g)$ be an admissible $3$-manifold as in Definition \ref{def of admissible}, so that $M \subset \subset \R\times M_0$ and $g =  c(e\oplus g_0)$ where $(M_0,g_0)$ is simple. Let $(x_1,x')$ be global coordinates in $\R \times M_0$, and let $\phi(x_1,x') = x_1$ be the natural LCW on $M$.

It is convenient to begin with a reduction to the case where $c \equiv 1$.  This can be done using the following lemma from ~\cite{KSaU}.

\begin{lemma} (\cite{KSaU}, Lemma 7.1)
Let $(M,g)$ be a compact Riemannian 3-manifold with smooth boundary, and let $\eps$ and $\mu$ be $C^2$ functions on $M$ with positive real part.  Let $c$ be any positive smooth function on $M$, and let $C_{g,\eps,\mu}^{\Gamma_1, \Gamma_2}$ represent the Cauchy data for $\eps, \mu$ with respect to the metric $g$.  Then
\[
C_{cg,\eps,\mu}^{\Gamma_1, \Gamma_2} = C_{g,c^{1/2}\eps, c^{1/2}\mu}^{\Gamma_1, \Gamma_2}.
\]
\end{lemma}
\begin{proof}
Lemma 7.1 in ~\cite{KSaU} is actually phrased in terms of the admittance map, which requires the additional condition that solutions to the Maxwell equations are unique.  However, the same proof works when the lemma is phrased in terms of the Cauchy data.  The key fact is that for a $k$-form $u$, 
\[
*_{cg} u = c^{3/2 - k} *_g u,
\]
so $(E,H)$ satisfies \eqref{maxwell equations} with metric $cg$ and coefficients $\eps$ and $\mu$ if and only if $(E,H)$ satisfies \eqref{maxwell equations} with metric $g$ and coefficients $c^{\half} \eps$ and $c^{\half} \mu$.  
\end{proof}

Thus it suffices to prove Theorem \ref{main theorem} in the case where $c \equiv 1$. We will from now on assume that $g = e \oplus g_0$ where $(M_0,g_0)$ is a simple $2$-manifold, and that we are working with the LCW $\phi(x) = x_1$. We will also assume that $(\hat{M}_0,g_0)$ is another simple $2$-manifold which is slightly larger than $M_0$.

As indicated in Lemma \ref{integral identity with boundary conditions}, we will need to construct two types of complex geometrical optics solutions. First we need sufficiently regular solutions to the Maxwell equations, with no restrictions on their boundary values. By Section \ref{sec:sec_reductions} it suffices to construct solutions $Z$ to the equation $(-\Delta + Q) Z = 0$ in $M$â such that $Y = (P-W^t) Z$ satisfies $Y^0 = Y^3 = 0$. Such solutions were obtained in \cite[Theorem 6.1a]{KSaU}. Second, we need solutions of the Dirac equation $(P+W^*)Y = 0$ satisfying $t Y |_\Gamma= 0$. Here $\Gamma \subset \partial M$ is meant to be a open neighbourhood of $\partial M \setminus \tilde{F}$, where $\tilde{F}$ is as in the statement of Theorem \ref{main theorem}.  By proper choice of $\Gamma$, we can ensure that 
\begin{equation}\label{Gamma inequality}
\langle d\phi, \nu \rangle < 0 \mbox{ on } \overline{\Gamma}.
\end{equation}

The next theorem states the existence and basic properties of both types of complex geometrical optics solutions.

\begin{thm}
\label{solution to maxwell}
Let $\eps, \mu \in C^3(M)$ have positive real parts, and let $W$ and $Q$ be as in Section \ref{sec:sec_reductions}. Fix $p \in \hat{M}_0 \setminus M_0$, and let $(r,\theta)$ be polar normal coordinates in $(\hat{M}_0,g_0)$ with center $p$, so that $(x_1,r,\theta)$ are global coordinates near $M$. Let $s_0, t_0$ be real constants, let $\lambda > 0$, and let $b(\theta) \in C^{\infty}(S^1)$.
\begin{enumerate}
\item[(a)]
For $\abs{\tau}$ sufficiently large and outside a countable subset of $\mR$, one can construct solutions to the equation $(-\Delta + Q)Z =0$ in $M$ satisfying
\[ (P + W) Y = 0, \qquad Y= (P-W^t) Z,\]
\[Y^0 = Y^3 = 0\]
where $Z \in H^3(M, \Lambda M)$ has the form
\begin{eqnarray}
\label{form of the maxwell solution}
Z = e^{-\tau(x_1 + ir)} |g(r,\theta)|^{-1/4}e^{i\lambda(x_1 + ir)} b(\theta)\begin{pmatrix} s_0\\0\\\hline t_0 *1\\0\end{pmatrix} +e^{-\tau(x_1 + ir)} R
\end{eqnarray}
with $\|R\|_{H^s} \leq C \tau^{s -1}$ for $0\leq s\leq 2$.
\item[(b)]
Let $\Gamma \subset \partial M$ be an open set so that \eqref{Gamma inequality} holds for $\phi(x) = x_1$. For $\tau > 0$ sufficiently large, there exists a solution $Z \in H_{\Delta}$ to $(-\Delta+\hat{Q})Z = 0$ in $M$ so that 
\[(P+W^*)Y = 0, \qquad Y= (P-\bar{W}) Z,\] and $Y$ has the form 
\begin{eqnarray}
\label{form of dirac solution}
Y = e^{\tau (x_1 + ir)} |g(r,\theta)|^{-1/4}\begin{pmatrix} s_0\\-is_0 dx^1\wedge dr\\\hline t_0 * 1\\ it_0*dx^1 \wedge dr  \end{pmatrix}  + e^{\tau (x_1 + ir)} (\tilde R +R')
\end{eqnarray}
with $\|\tilde R\|_{H^{-1}} \leq C\tau^{-3/2}$, $\|R'\|_{L^2}\leq C\tau^{-1/2}$, and $Y$ satisfies the boundary condition \[t Y|_{\Gamma} = 0.\]
\end{enumerate}
\end{thm}
\begin{proof}[Proof of Theorem \ref{solution to maxwell}(a)]
It is proved in \cite[Theorem 6.1a]{KSaU} that there is a solution $Z \in H^2(M, \Lambda M)$ of this form. That theorem only states $L^2$ bounds for $R$, but the proof of the theorem relies on \cite[Proposition 5.1]{KSaU} which actually gives the $H^s$ bounds for $0 \leq s \leq 2$. It remains to show that under the assumption $\eps, \mu \in C^3(M)$ (which implies $Q \in C^1$) one has in fact $Z \in H^3(M, \Lambda M)$. This follows from interior elliptic regularity since the solution in \cite[Theorem 6.1a]{KSaU} is constructed in a neighbourhood of $M$, and $Z$ satisfies $-\Delta Z = -QZ$ where $QZ \in H^1$ in this neighbourhood.
\end{proof}

In the remainder of this section we will prove Theorem \ref{solution to maxwell}(b). This will be done by constructing a solution for the equation $(-\Delta + \hat Q)Z =0$
of the form 
\begin{eqnarray}
\label{form of Z}
Z = e^{\tau x_1}(A + R), \qquad tZ|_\Gamma = t(\delta + i\bar W)Z|_\Gamma = 0,
\end{eqnarray}
for some appropriate amplitude $A$ chosen so that 
\[e^{-\tau x_1}(-\Delta +\hat Q)e^{\tau x_1}A = O_{L^2(M)}(1).\]
Then $Y = (P-\bar W) Z$ solves $(P + W^*) Y = 0$ and satisfies the desired boundary conditions.

The remainder term $R$ in \eqref{form of Z} will be constructed using the following Carleman estimate for graded forms with suitable boundary conditions.  If $u$ is a graded form we write $u_\perp = \nu \wedge i_\nu u $ and $u_{||} = u - u_\perp$. The proof of the following estimate was first given in \cite{CST}. Here $\|\cdot\| = \|\cdot\|_{M}$ and $\|\cdot\|_{\partial M}$ are the relevant $L^2$ norms in $M$ and $\partial M$.

\begin{prop}
\label{carleman estimates} {\rm (\cite[Theorem 7.2]{CST})} 
Let $(M,g)$ be an admissible manifold as in Definition \ref{def of admissible} and $\phi(x) = x_1$ a LCW. Let $\hat{Q}$ and $\sigma$ be $L^\infty$ endomorphisms on the space of graded forms. There is $C > 0$ such that for all smooth graded forms $u$ satisfying
\[te^{\tau\phi}\delta e^{-\tau\phi}u + t\sigma u_\perp|_{\Gamma} = tu|_{\Gamma} =0,\]
\[u|_{\partial M \backslash\Gamma} = \nabla_\nu u |_{\partial M \backslash\Gamma} =0,\]
and for all $\tau \geq C$, we have the following estimate
\[\tau \|u\| + \|\nabla u\| + \tau^{3/2} \|ti_\nu u\|_{\partial M} + \sqrt\tau \|\nabla' ti_\nu u\|_{\partial M} + \sqrt\tau \|t\nabla_\nu u_{||}\|_{\partial M} \leq C \|e^{\tau \phi} (-\Delta + \hat{Q}) e^{-\tau \phi} u\|.\]
\end{prop}

Proposition \ref{carleman estimates} relies on the fact that \eqref{Gamma inequality} holds on $\Gamma$ -- this is one of the key reasons for choosing $\tilde{F}$ as we did in Theorem \ref{main theorem}.    

\subsection{Construction of the amplitude}
 If $\hat Q$ is an $L^{\infty}$ endomorphism of $\Lambda M$, we look for a solution $Z$ of 
$$
(-\Delta + \hat Q)Z = 0 \quad \text{in } M
$$
having the form 
$$
Z = e^{\tau x_1}(A + R)
$$
where $\tau > 0$ is large. To this end we will construct an amplitude $A$ such that 
\[
\| e^{-\tau x_1} \Delta e^{\tau x_1} A\|_{M}, \|A\|_{M} = O_{L^2(M)}(1).
\]
Recall that since $(M,g)$ is an admissible manifold and we have reduced to the case $c=1$, the metric can be written globally as $g = e \oplus g_0$ for some conformal factor $c$ and metric $g_0$ on the simple 2-dimensional manifold $M_0$. 

We now construct the amplitude $A$ mentioned above.

\begin{prop} 
\label{amplitude}
For any smooth functions $b^j, b^j_r, b^j_\theta \in C^\infty (S^{1})$ there exist graded forms $ A = A^0+ A^1+ *B^1+*B^0$ satisfying 
\[ \|e^{-\tau x_1} \Delta  e^{\tau x_1}A\|_M = O(1),\ \ \|A\|_M = O(1)\]
of the form
\begin{eqnarray*}
&& A^0 = \abs{g_0}^{-1/4} e^{i\tau r}b^0(\theta),\\
&& A^1 = \abs{g_0}^{-1/4} e^{i\tau r}b^1(\theta) dx_1 + e^{i\tau r}(\abs{g_0}^{-1/4} b_r^1(\theta)\,dr + \abs{g_0}^{1/4} b^1_\theta(\theta) \,d\theta),\\
&& B^0 = \abs{g_0}^{-1/4} e^{i\tau r}b^3(\theta),\\
&& B^1 = \abs{g_0}^{-1/4} e^{i\tau r}b^2(\theta) dx_1 + e^{i\tau r}(\abs{g_0}^{-1/4} b^2_r(\theta)\,dr + \abs{g_0}^{1/4} b^2_\theta(\theta) \,d\theta).
\end{eqnarray*} 
\end{prop}

\begin{proof}
Any $1$-form $A^1$ on $M$ may be written as 
$$
\tilde A^1 =  A_1^1 \,dx^1 + (A^1)'
$$
with $A_1^1 = \langle A, dx^1 \rangle$ and $(A^1)'$ is a $1$-form on $M_0$ that contains $x_1$ as a parameter. If $A^0$ is a $0$-form and $A^1$ is a $1$-form, we have 
\begin{gather*}
\Delta_{g} A^0 =(\partial_1^2  + \Delta_{x'} ) A^0, \\
\Delta_{ g} A^1= (\Delta_{g} A_1^1) \,dx^1 + \partial_1^2 (A^1)' + \Delta_{x'}(A^1)'
\end{gather*}
where $\Delta_{x'}$ is the Hodge Laplacian on $(M_0,g_0)$.

We compute the conjugated Hodge Laplacian acting on $0$-forms and $1$-forms, 
\begin{eqnarray*}
e^{-\tau x_1} (-\Delta_g ) e^{\tau x_1} A^0 &=& (-\partial_1^2 - 2 \tau \partial_1 - \tau^2 - \Delta_{x'})A^0, \\
e^{-\tau x_1} (-\Delta_{g} ) e^{\tau x_1} A^1 &=& \left[ (-\partial_1^2 - 2 \tau \partial_1 - \tau^2 - \Delta_{x'}) A^1_1 \right] \,dx^1 \\
& & + (-\nabla_{\partial_1}^2 -2 \tau \nabla_{\partial_1} - \tau^2 - \Delta_{x'}) (A^1)'.
\end{eqnarray*}
We now make the simplifying assumption that $A^0$ and $A^1$ only depend on $x'$. Then the above expressions simplify to 
\begin{gather*}
e^{-\tau x_1} (-\Delta_{ g} ) e^{\tau x_1} A^0= (-\Delta_{x'} - \tau^2) A^0, \\
e^{-\tau x_1} (-\Delta_{ g} ) e^{\tau x_1} A^1= \left[ (-\Delta_{x'} - \tau^2) A^1_1 \right] \,dx^1 + (-\Delta_{x'}-\tau^2) (A^1)'.
\end{gather*}
It is enough to arrange that 
\begin{eqnarray}
\label{equation for A's}
\norm{(-\Delta_{x'} - \tau^2) A^0}_{M_0} = O(1), \quad \norm{A^0}_{M_0} = O(1), \\ \nonumber
\norm{(-\Delta_{x'} - \tau^2) A^1_1}_{M_0} = O(1), \quad \norm{A^1_1}_{M_0} = O(1), \\ \nonumber
\norm{(-\Delta_{x'} - \tau^2) (A^1)'}_{M_0} = O(1),  \quad \norm{(A^1)'}_{M_0} = O(1).
\end{eqnarray}

The construction for the scalar parts $A^0$ and $A^1_1$ is identical: we look for a WKB type quasimode 
$$
A^0 = e^{i\tau \psi} a
$$
and note that 
$$
(-\Delta_{x'} - \tau^2) A^0 = e^{i\tau \psi} \left[ \tau^2 (\abs{d\psi}^2 - 1) a- i\tau ( 2\nabla_{\nabla \psi} a + (\Delta \psi) a) - \Delta a\right].
$$
In order to satisfy \eqref{equation for A's}, we need $\psi$ and $a$ to satisfy the eikonal equation $ \abs{d\psi}^2 - 1 = 0$, and the transport equation $2\nabla_{\nabla \psi} a + (\Delta \psi) a = 0$. To find $\psi$, let $(\hat{M}_0,g_0)$ be a slightly larger simple manifold than $(M_0,g_0)$, let $p \in \hat{M}_0 \setminus M_0$, let $(r,\theta)$ be polar normal coordinates in $(\hat{M}_0,g_0)$ with center $p$, and choose 
$$
\psi(r,\theta) = r.
$$
In these coordinates, the metric has the form 
$$
g_0(r,\theta) = \left( \begin{array}{cc} 1 & 0 \\ 0 & m(r,\theta) \end{array} \right)
$$
for some smooth positive function $m$. It follows that 
$$
\abs{d\psi}^2 = 1.
$$
We also have $\nabla \psi = \partial_r$ and $\Delta \psi = \frac{1}{2} \partial_r(\log\abs{g_0})$ where $\abs{g_0(r,\theta)}$ is the determinant of $g_0$. The transport equation becomes 
$$
2 \partial_r a +  \frac{1}{2} \partial_r(\log\abs{g_0}) a= 0
$$
and it has the solution 
$$
a= \abs{g_0}^{-1/4} b(\theta)
$$
where $b$ can be any function in $C^{\infty}(S^{1})$. Therefore, setting $A_0 = |g_0|^{-1/4}e^{i\tau r} b^0(\theta)$, the first line of \eqref{equation for A's} is satisfied. A similar choice for $A_1^1$, substituting $b^1$ for $b^0$, would satisfy the second line of \eqref{equation for A's}.

Moving on to $(A^1)'$, we look for an ansatz of the form 
$$
(A^1)' = e^{i\tau \psi} a', \quad a' = a_r \,dr + a_{\theta} \,d\theta.
$$
We have 
$$
(-\Delta_{x'} - \tau^2) (A^1)'= e^{i\tau \psi} \left[ \tau^2 (\abs{d\psi}^2 - 1) a' - i\tau [ 2\nabla_{\nabla \psi} a' + (\Delta \psi) a'] - \Delta a' \right].
$$
Choosing the same function $\psi(r,\theta) = r$ as above we have $\abs{d\psi}^2=1$. Since 
$$
\nabla_{\partial_r} dr = 0, \quad \nabla_{\partial_r} d\theta = -\frac{1}{2} \partial_r(\log\abs{g_0}) \,d\theta,
$$
the transport equations for $a_r$ and $a_{\theta}$ become 
\begin{align*}
2 \partial_r a_r +  \frac{1}{2} \partial_r(\log\abs{g_0}) a_r &= 0, \\
2 \partial_r a_{\theta} - \frac{1}{2} \partial_r(\log\abs{g_0}) a_{\theta} &= 0.
\end{align*}
These have the solutions 
$$
a_r = \abs{g_0}^{-1/4} b_r(\theta), \quad a_{\theta} = \abs{g_0}^{1/4} b_{\theta}(\theta)
$$
for any $b_\theta, b_r \in C^\infty(S^1)$.
Setting $\tilde A' = e^{i\tau r}(\abs{g_0}^{-1/4} b^1_r(\theta)dr +  \abs{g_0}^{1/4} b^1_{\theta}(\theta) d\theta)$, the third line of \eqref{equation for A's} is satisfied.

Now we can use the fact that $*$ commutes with $\Delta$ to construct $B^0$ and $B^1$, with the same asymptotics, and finish the proof.
\end{proof}

\subsection{Construction of the remainder}

We now construct the remainder term $R$ with the appropriate boundary conditions and asymptotics in $\tau$. More precisely we need to construct $R$ to satisfy
\[ e^{-\tau x_1}(-\Delta + \hat Q) e^{\tau x_1} R = -  e^{-\tau x_1} (-\Delta+\hat Q) e^{\tau x_1} A\]
with boundary conditions
\[t R |_\Gamma = -t A |_\Gamma, \qquad t \big(e^{-\tau x_1}\delta(e^{\tau x_1} R) + i \bar{W} R\big) |_\Gamma= -t\big(e^{-\tau x_1}\delta(e^{\tau x_1} A) + i \bar{W} A \big) |_\Gamma.\]
The first step in constructing such an $R$ is the following solvability result similar to \cite[Proposition 7.3]{CST}, obtained from Proposition \ref{carleman estimates}. Here we write $\Delta_{\phi} = e^{\tau \phi} \Delta e^{-\tau \phi}$.

\begin{prop}
\label{construction of remainder}
Let $\hat{Q}$ and $W$ be $C^1$ endomorphisms on $M$ acting on graded forms. If $\tau > 0$ is large, for any graded forms $F\in L^2(M, \Lambda M)$ and $f_1, f_2\in C^\infty(M, \Lambda M)$ there exists a solution $u \in L^2(M, \Lambda M) \cap e^{-\tau \phi} H_\Delta$ to\\
\begin{equation}\label{inhomogenous bvp}
\begin{split}
(-\Delta_{-\phi} + \hat{Q})u &= F \\
tu |_{\Gamma} &= tf_1 |_{\Gamma},\\
(e^{-\tau \phi}t\delta e^{\tau \phi}u +tW u_\perp)|_{\Gamma} &= (e^{-\tau\phi}t \delta e^{\tau \phi} f_2 +t W (f_{2})_{\perp})|_{\Gamma}
\end{split}
\end{equation}
and which satisfies the estimate
\begin{eqnarray}
\label{estimate by RHS}
\|u\|_M \leq \frac{C}{ \tau}\left( \|F\|_{M} + \frac{1}{\sqrt{\tau}}\|e^{-\tau\phi}t\delta e^{\tau\phi}f_2 + t W (f_2)_{\perp} \|_{\partial M} + \sqrt\tau \|t f_1\|_{\partial M} \right).
\end{eqnarray}
\end{prop}

Note that the boundary conditions in \eqref{inhomogenous bvp} make sense because $u \in e^{-\tau \phi} H_\Delta$ guarantees that the traces in \eqref{inhomogenous bvp} are well defined; see Lemma \ref{lemma_hdelta_properties}.
  
\begin{proof}
Let ${ H}$ be the subspace of $L^2(M, \Lambda M)$ defined by
\[{ H} := \{(-\Delta_{\phi} + \hat{Q}^*) v \mid v \in C^\infty,\ tv =0,\ e^{\tau \phi}t\delta (e^{-\tau \phi}v) = -ti_\nu W^* i_\nu v,\ v |_{\Gamma^c} = \nabla_N v |_{\Gamma^c} =0\}\]
Define on $H$ the linear operator 
\[
L((-\Delta_{\phi} + \hat{Q}^*) v) := (v|F)_{M} + (e^{\tau\phi} i_\nu d e^{-\tau\phi} v | t f_1)_{\Gamma} + (i_\nu v | e^{-\tau\phi}t\delta (e^{\tau \phi}f_2) + t W (f_2)_\perp)_{\Gamma}.
\]
We see that $L$ is well defined and bounded, since by Proposition \ref{carleman estimates} and by the fact that $t i_{\nu} dv = -\nabla' t i_{\nu} v$ for $v$ as above \cite[Lemma 3.4]{CST}, the quantity $|L(w)|$ is bounded by 
\begin{multline*}
\|v\|_{M} \|F\|_{M} + \frac{\| t e^{\tau\phi} i_\nu d e^{-\tau\phi} v\|_{\Gamma}}{\sqrt{\tau}} \sqrt{\tau}\|t f_1\|_{\Gamma} + \sqrt{\tau} \| i_\nu v\|_{\Gamma}\frac{\|e^{-\tau\phi}t\delta (e^{\tau \phi}f_2) + t W (f_2)_\perp\|_{\Gamma}}{\sqrt{\tau}}\\
\leq \frac{C}{{\tau}} \|(-\Delta_{\phi} + \hat{Q}^*) v\|_{M} \left( \|F\|_{M} +  \sqrt{\tau}\|t f_1\|_{\Gamma} + \frac{1}{ \sqrt{\tau}}\|e^{-\tau\phi}t\delta (e^{\tau \phi}f_2) + t W (f_2)_\perp\|_{\Gamma} \right).
\end{multline*}
By Hahn-Banach, there exists $u\in L^2$ with the estimate 
\[
\|u\|_{M} \leq \frac{C}{{\tau}}  \left( \|F\|_{M} +  \sqrt{\tau}\|t f_1\|_{\Gamma} + \frac{1}{ \sqrt{\tau}}\|e^{-\tau\phi}t\delta (e^{\tau \phi}f_2) + t W (f_2)_\perp\|_{\Gamma} \right)
\]
satisfying
\begin{multline*}
((-\Delta_{\phi} + \hat{Q}^*) v|u) = L((-\Delta_{\phi} + \hat{Q}^*) v) \\
                            =  (v|F)_{M} + (i_\nu d(e^{-\tau\phi} v) | t(e^{\tau\phi} f_1))_{\Gamma} + (i_\nu (e^{-\tau\phi}v) | t\delta (e^{\tau \phi}f_2) + e^{\tau \phi} tW(f_2)_{\perp})_{\Gamma}
\end{multline*}
for all $v$ as in the definition of $H$.

We consider first those $v$ which are compactly supported in the interior of $M$ to see that $u$ solves $(-\Delta_{-\phi} + \hat{Q}) u= F$. Thus in particular $u \in e^{-\tau \phi}H_\Delta$, so the traces of $e^{\tau \phi} u$ are well defined by Lemma \ref{lemma_hdelta_properties}. Now we can integrate by parts using \eqref{iparts1}, \eqref{iparts2}, the relation $e^{\tau \phi}t\delta (e^{-\tau \phi}v) = -i_\nu W^* i_\nu v$, and the fact that $v$ vanishes to first order on $\Gamma^c$, to see that on the boundary $u$ satisfies
\begin{equation*}
\begin{split}
&(te^{\tau\phi} u|i_\nu d(e^{-\tau\phi}v))_{\Gamma} + (te^{-\tau\phi} \delta (e^{\tau\phi}u) |ti_\nu v)_{\Gamma} + (i_\nu u |i_\nu W^* i_\nu v)_{\Gamma} \\
= \quad & (te^{\tau\phi} f_1|i_\nu d(e^{-\tau\phi}v))_{\Gamma} +(te^{-\tau\phi}\delta (e^{\tau \phi}f_2) + t W (f_2)_\perp|t  i_\nu v)_{\Gamma}
\end{split}
\end{equation*}
which becomes 
\begin{equation*}
\begin{split}
&(te^{\tau\phi} u|i_\nu d(e^{-\tau\phi}v))_{\Gamma} + (te^{-\tau\phi} \delta (e^{\tau\phi}u) + t W u_\perp|ti_\nu v)_{\Gamma} \\
= \quad &(te^{\tau\phi} f_1|i_\nu d(e^{-\tau\phi}v))_{\Gamma} +(te^{-\tau\phi}\delta (e^{\tau \phi}f_2) + t W (f_2)_\perp|t  i_\nu v)_{\Gamma}
\end{split}
\end{equation*}
for all $v$ as in the definition of $H$.
Now for all $\psi \in \Omega^{k+1}(M)$ with $\mathrm{supp}(i_{\nu} \psi) \subset \Gamma$, we may choose $v \in \Omega^k(M)$ such that 
\[v|_{\partial M} = 0,\ t\delta (e^{-\tau\phi}v) =0,\ i_\nu d(e^{-\tau\phi} v) = i_\nu\psi\]
(See Lemma 3.3.2 of \cite{gunther}.) A brief computation shows that $v$ vanishes to first order on $\Gamma^c$ and can be used in the above identity, which gives 
\[
(te^{\tau\phi} u|i_\nu \psi)_{\Gamma} = (te^{\tau\phi} f_1|i_\nu \psi)_{\Gamma}
\]
whenever $\mathrm{supp}(i_{\nu} \psi) \subset \Gamma$. Therefore we see that $te^{\tau\phi} u = te^{\tau\phi} f_1$ on $\Gamma$ and 
\[
(te^{-\tau\phi} \delta (e^{\tau\phi}u)+ t W u_{\perp}|t i_\nu v)_{\Gamma}=  (te^{-\tau\phi} \delta (e^{\tau\phi}f_2)+ tW(f_2)_{\perp}|t i_\nu v)_{\Gamma}
\]
for all $v$ as in the definition of $H$. Applying an argument similar to \cite[Lemma 3.3.2]{gunther} (see also Lemma \ref{LemmaofRequirement}) again allows us to conclude that 
\[e^{-\tau\phi}t\delta (e^{\tau\phi}u) + tWu_\perp = te^{-\tau\phi}\delta (e^{\tau\phi}f_2) + tW (f_2)_\perp\] on $\Gamma$.
\end{proof}
This gives the following existence result for complex geometrical optics solutions with boundary conditions.

\begin{prop}
\label{CGO vanishing on Gamma}
Let $\hat{Q}$ and $W$ be $C^1$ endomorphisms acting on graded forms. For $\tau>0$ sufficiently large there exist solutions to the equation 
\[(-\Delta + \hat{Q}) Z =0\]
of the form 
\[Z = e^{\tau x_1}(A +R),\]
with boundary conditions
\[tZ|_\Gamma = 0,\ t(\delta Z + i\bar{W} Z_{\perp}) |_\Gamma = 0\]
 where $A$ is as in Proposition \ref{amplitude} with $R \in L^2 \cap e^{-\tau x_1}H_\Delta$ satisfying the estimate $\|R\|\leq \frac{C}{\sqrt \tau}$. 
\end{prop}
\begin{proof}
By Proposition \ref{amplitude}, the $L^2$ norm of $F :=  e^{-\tau x_1}(-\Delta + \hat{Q})e^{\tau x_1}A$ is $O(1)$ as $\tau \to \infty$.  Now apply Proposition \ref{construction of remainder} to obtain $R \in L^2 \cap e^{-\tau x_1}H_\Delta$ solving 
\begin{eqnarray*}
e^{-\tau x_1}(-\Delta + \hat Q) e^{\tau x_1} R &=& -F,\\
            t R |_\Gamma &=& -tA|_\Gamma,\\
t(e^{-\tau x_1}\delta e^{\tau x_1} R +  i\bar{W} R_\perp)|_\Gamma &=& -t (e^{-\tau x_1}\delta e^{\tau x_1} A +  i\bar{W} A_\perp)|_\Gamma \\
\end{eqnarray*}
with the estimates
\begin{eqnarray*}
\|R\|_M &\leq& \frac{C}{ \tau} \left( \|F\|_{M} + \sqrt \tau \| t A\|_{\partial M} + \frac{1}{\sqrt\tau} \| e^{-\tau x_1} t \delta e^{\tau x_1} A + t  i\bar{W} A_\perp\|_{\partial M} \right) \\
&\leq& \frac{C}{\sqrt \tau}.
\end{eqnarray*}
\end{proof}

\begin{proof}[Proof of Theorem \ref{solution to maxwell}(b)]
By Proposition \ref{CGO vanishing on Gamma}, there is a solution $Z_0 \in H_{\Delta}$ of the equation $(-\Delta + \hat Q)Z_0 =0$ satisfying 
\[
Z_0 = e^{\tau x_1}(A + R), \qquad tZ_0|_\Gamma = t(\delta + i\bar W)Z_0|_\Gamma = 0
\]
where the amplitude $A$ is chosen to be 
\[
A = e^{i\tau r} \abs{g}^{-1/4} b(\theta) \zeta_2
\]
with $\zeta_2 = -i(s_0 \,dx^1+ t_0 * dx^1)$ where $s_0, t_0$ are real constants.

We define $Z := \frac{1}{\tau} Z_0$, so $Z$ will solve the same equation as $Z_0$ with the same boundary conditions. Then $Y := (P-\bar{W})Z$ will solve $(P+W^*)Y = 0$ in the sense of distributions. Since $tZ|_{\Gamma} = 0$, we have $Z|_{\Gamma} = Z_{\perp}|_{\Gamma}$ and $tdZ|_{\Gamma} = 0$ by Lemma \ref{Trace}. Consequently 
\[
tY|_{\Gamma} = \frac{1}{i} (tdZ- t\delta Z - i t \bar{W} Z)|_{\Gamma} = -\frac{1}{i} t(\delta Z + i \bar{W} Z_{\perp})|_{\Gamma} = 0.
\]
Finally, we observe that  
\begin{align*}
Y &= \frac{1}{\tau}(P- \bar W)  (e^{\tau (x_1+ir)} \abs{g}^{-1/4} b(\theta) \zeta + e^{\tau x_1} R) \\
&= \frac{1}{\tau} P ( e^{\tau (x_1+ir)  }\abs{g}^{-1/4} b(\theta) \zeta) 
 - \frac{1}{\tau} \bar W e^{\tau (x_1 +ir) } \abs{g}^{-1/4} b(\theta) \zeta+  \frac{1}{\tau} [P, e^{\tau x_1}] R \\
 & \qquad + e^{\tau x_1}\underbrace{\frac{1}{\tau} (P-\bar W) (R)}_{\tilde R}.
\end{align*}
A direct computation (see \cite[Proof of Theorem 6.1(b)]{KSaU}) gives the result.
\end{proof}

\section{Recovery of coefficients}

Here we use the solutions constructed in Theorem \ref{solution to maxwell} to recover the coefficients of the Maxwell equations from partial boundary measurements. The argument is essentially the same as in \cite{KSaU}.

\begin{proof}[Proof of Theorem \ref{main theorem}]
Let $Z_1$ be the solution to $(-\Delta+Q_1) Z_1 = 0$ constructed in Theorem \ref{solution to maxwell}(a) having the form
\[ Z_1 = e^{-\tau(x_1 + ir)} |g|^{-1/4}e^{i\lambda(x_1 + ir)} b(\theta)\begin{pmatrix} s_0\\0\\\hline t_0 *1\\0\end{pmatrix} +e^{-\tau(x_1 + ir)} R.
\]
Choose an open set $F_1\subset \partial M$ such that $F_{\varphi} \subset F_1 \subset \subset \tilde{F}$, and let $\Gamma = \partial M \setminus \overline{F}_1$. Then \eqref{Gamma inequality} holds, and we can take $Y_2$ to be the solution to $(P+W_2^*)Y_2 = 0$ constructed in Theorem \ref{solution to maxwell}(b), satisfying $tY_2|_{\Gamma} = 0$ and having the form 
\[Y_2 = e^{\tau (x_1 - ir)} |g|^{-1/4}\begin{pmatrix} s_0\\is_0 dx^1\wedge dr\\\hline t_0 * 1\\ -it_0*dx^1 \wedge dr  \end{pmatrix}  + e^{\tau (x_1 - ir)} (\tilde R +R').\]
(In fact we solve $(P-\bar{W}_2^*) Y = 0$ where $Y$ is as in Theorem \ref{solution to maxwell}(b), and we let $Y_2 = \bar{Y}$.) Here $s_0, t_0, \lambda$ are any real constants with $\lambda > 0$, $b(\theta)$ is any function in $C^{\infty}(S^1)$, and $\tau > 0$ can be taken arbitrarily large.

By taking $\Gamma_2 = \tilde F$ in Lemma \ref{integral identity with boundary conditions} we have $\mathrm{supp}(tY_2) \subset \Gamma_2$. Thus we have the integral identity 
\begin{equation*}
((Q_1-Q_2)Z_1|Y_2)= \int_M\langle (Q_1 - Q_2) Z_1 | \overline{Y_2}\rangle \,dV = 0.
\end{equation*}
Writing this out term by term we have 
\begin{multline}
\label{boundary integral id with CGO}
\int_M  \Bigg\langle (Q_1-Q_2) e^{i\lambda(x_1+ir)} b(\theta)  \begin{pmatrix} s_0\\0\\\hline t_0 *1\\0\end{pmatrix} |  \begin{pmatrix} s_0\\ -is_0 dx^1\wedge dr\\\hline t_0 * 1\\ it_0*dx^1 \wedge dr  \end{pmatrix} \Bigg\rangle \\
 + \int_M  \Bigg\langle (Q_1-Q_2)R |(\tilde R + R')  \Bigg\rangle \abs{g}^{1/2}
  + \int_M  \Bigg\langle (Q_1-Q_2)e^{i\lambda(x_1+ir)} b(\theta)\begin{pmatrix} s_0\\0\\\hline t_0 *1\\0\end{pmatrix}  | \tilde R + R'  \Bigg\rangle \abs{g}^{1/4} \\
  +\int_M  \Bigg\langle (Q_1-Q_2)R|  \begin{pmatrix} s_0\\ -is_0 dx^1\wedge dr\\\hline t_0 * 1\\ it_0*dx^1 \wedge dr  \end{pmatrix}   \Bigg\rangle \abs{g}^{1/4} =0.
\end{multline}
Here all integrals are with respect to the $3$-form $\abs{g}^{-1/2} \,dV = \,dx_1 \,dr \,d\theta$.

We will argue that all terms in (\ref{boundary integral id with CGO}) except for the first one vanish as $\tau \to\infty$. Indeed, by Theorem \ref{solution to maxwell}, $\|R\|_{H^s} \leq C \tau^{s-1}$, $\|\tilde R\|_{H^{-1}} \leq C\tau^{-3/2}$, and $\|R'\|_{L^2}\leq C\tau^{-1/2}$. Using these estimates for $s=1$ and $s=0$ we see that the second term is $O(\tau^{-3/2})$. Here we use that $\eps_1=\eps_2$ and $\mu_1 = \mu_2$ to second order on $\partial M$, to have $Q_1=Q_2$ on $\partial M$. By the same estimates for $s=0$ we see that the third term is $O(\tau^{-1/2})$ and the fourth term is $O(\tau^{-1})$. Thus taking $\tau \to \infty$ in (\ref{boundary integral id with CGO}) implies that 
\[\int_M  \Bigg\langle (Q_1-Q_2) \begin{pmatrix} s_0\\0\\\hline t_0 *1\\0\end{pmatrix} |  \begin{pmatrix} s_0\\ -is_0 dx^1\wedge dr\\\hline t_0 * 1\\ it_0*dx^1 \wedge dr  \end{pmatrix} \Bigg\rangle e^{i\lambda(x_1+ir)} b(\theta) \,dx_1 \,dr \,d\theta=0\]
for all real $s_0$, $t_0$, $\lambda$ with $\lambda > 0$, and all smooth functions $b(\theta) \in C^\infty(S^{1})$.

The rest of the proof is identical with the proof of Theorem 1.1 in \cite{KSaU}; we give the argument for completeness. Let $q_{\alpha}$ and $q_{\beta}$ be the elements of $Q_1-Q_2$, interpreted as a $8 \times 8$ matrix, which correspond to the $(1,1)$th and $(5,5)$th elements, respectively. By Lemma \ref{lemma:reduction_schrodinger}
\begin{align*}
q_{\alpha} &= \frac{1}{2}\Delta(\alpha_1-\alpha_2) + \frac{1}{4} \langle d\alpha_1,d\alpha_1 \rangle - \frac{1}{4} \langle d\alpha_2,d\alpha_2 \rangle - \omega^2(\eps_1 \mu_1 - \eps_2 \mu_2), \\
q_{\beta} &= \frac{1}{2}\Delta(\beta_1-\beta_2) + \frac{1}{4} \langle d\beta_1,d\beta_1 \rangle - \frac{1}{4} \langle d\beta_2,d\beta_2 \rangle - \omega^2(\eps_1 \mu_1 - \eps_2 \mu_2).
\end{align*}
With the two choices $(s_0,t_0) = (1,0)$ and $(s_0,t_0) = (0,1)$, the special form of $Q_1$ and $Q_2$ in Lemma \ref{lemma:reduction_schrodinger} shows that we obtain the two identities 
\begin{eqnarray*}
 & \int_M e^{i\lambda(x_1+ir)} b(\theta) q_{\alpha}(x) \,dx_1 \,dr \,d\theta = 0, & \\
 & \int_M e^{i\lambda(x_1+ir)} b(\theta) q_{\beta}(x) \,dx_1 \,dr \,d\theta = 0. & 
\end{eqnarray*}

Let $T = \mR \times M_0$. Since $Q_1|_{\partial M} = Q_2|_{\partial M}$, the zero extensions of $q_{\alpha}$ and $q_{\beta}$ to $T \setminus M$ are continuous functions and the integrals above may be taken over $T$. Varying $b(\theta)$, it follows that for all $\theta$ we have 
\begin{equation*}
\int_0^{\infty} e^{-\lambda r} \left[ \int_{-\infty}^{\infty} e^{i\lambda x_1} q_{\alpha}(x_1,r,\theta) \,dx_1 \right] \,dr = 0
\end{equation*}
and similarly for $q_{\beta}$. Since $(r,\theta)$ are polar normal coordinates in $M_0$, the curves $r \mapsto (r,\theta)$ are geodesics in $M_0$. Denoting the expression in brackets by $f_{\alpha}(r,\theta)$ and varying the center of polar normal coordinates in Theorem \ref{solution to maxwell} and varying $\theta$, we obtain that 
\begin{equation*}
\int_0^{\infty} f_{\alpha}(\gamma(r)) \exp\left[ -\int_0^r \lambda \,ds \right] \,dr = 0
\end{equation*}
for all geodesics $\gamma$ in $M_0$ which begin and end at points of $\partial M_0$. This shows the vanishing of the geodesic ray transform of the function $f_{\alpha}$ with constant attenuation $-\lambda$. Since this transform is injective on simple two-dimensional manifolds \cite{SaU}, we have $f_{\alpha} \equiv 0$ for all $\lambda$. Thus 
\begin{equation*}
\int_{-\infty}^{\infty} e^{i\lambda x_1} q_{\alpha}(x_1,r,\theta) \,dx_1 = 0
\end{equation*}
for all $\lambda > 0$, $r$ and $\theta$. Uniqueness for the Fourier transform in $x_1$ shows that $q_{\alpha} \equiv 0$ in $M$. We obtain $q_{\beta} \equiv 0$ in $M$ by the exact same argument.

We have arrived at the following two equations in $M$:
\begin{eqnarray*}
 & -\frac{1}{2} \Delta(\alpha_1-\alpha_2) - \frac{1}{4}\langle d(\alpha_1+\alpha_2),d(\alpha_1-\alpha_2) \rangle + \omega^2(\eps_1 \mu_1 - \eps_2 \mu_2) = 0, & \\
 & -\frac{1}{2} \Delta(\beta_1-\beta_2) - \frac{1}{4}\langle d(\beta_1+\beta_2),d(\beta_1-\beta_2) \rangle + \omega^2(\eps_1 \mu_1 - \eps_2 \mu_2) = 0. & 
\end{eqnarray*}
Let $u = (\eps_1/\eps_2)^{1/2}$ and $v = (\mu_1/\mu_2)^{1/2}$. Then $\frac{1}{2}(\alpha_1-\alpha_2) = \log\,u$ and $\frac{1}{2}(\beta_1-\beta_2) = \log\,v$, and the equations become 
\begin{eqnarray*}
 & -\Delta(\log\,u) - (\eps_1 \eps_2)^{-1/2} \langle d(\eps_1 \eps_2)^{1/2}, d(\log\,u) \rangle + \omega^2(\eps_1 \mu_1 - \eps_2 \mu_2) = 0, & \\
 & -\Delta(\log\,v) - (\mu_1 \mu_2)^{-1/2} \langle d(\mu_1 \mu_2)^{1/2}, d(\log\,v) \rangle + \omega^2(\eps_1 \mu_1 - \eps_2 \mu_2) = 0. & 
\end{eqnarray*}
Multiplying the first equation by $(\eps_1 \eps_2)^{1/2}$ and the second by $(\mu_1 \mu_2)^{1/2}$ and using that $\delta(a\, dw) = -a \Delta w - \langle da, dw \rangle$, we obtain the equations 
\begin{eqnarray*}
 & \delta ((\eps_1 \eps_2)^{1/2} d (\log u)) + \omega^2 (\eps_1 \eps_2)^{1/2} \eps_2 \mu_2 (u^2 v^2 - 1) = 0, & \\
 & \delta ((\mu_1 \mu_2)^{1/2} d (\log v)) + \omega^2 (\mu_1 \mu_2)^{1/2} \eps_2 \mu_2 (u^2 v^2 - 1) = 0. & 
 \end{eqnarray*}
Since $d(\log u) = u^{-1} du$, $d(\log v) = v^{-1} dv$, we see that $u$ and $v$ satisfy the semilinear elliptic system 
\begin{eqnarray*}
 & \delta(\eps_2 d u) + \omega^2 \eps_2^2 \mu_2 (u^2 v^2 - 1)u = 0, & \\
 & \delta(\mu_2 d v) + \omega^2 \eps_2 \mu_2^2 (u^2 v^2 - 1)v = 0. &  
\end{eqnarray*}
The assumptions on $\eps_j$ and $\mu_j$ on $\partial M$ imply that $u = v = 1$ and $\partial_{\nu} u = \partial_{\nu} v = 0$ on $\partial M$. Also, the above equations imply that the pair $(\tilde{u},\tilde{v}) = (1,1)$ is a solution of the semilinear system in all of $M$. Unique continuation holds for this system (see for instance \cite[Appendix B]{KSaU}), and we obtain $u \equiv 1$ and $v \equiv 1$ in $M$. This proves that $\eps_1 \equiv \eps_2$ and $\mu_1 \equiv \mu_2$ in $M$ as required.
\end{proof}

\begin{proof}[Proof of Theorem \ref{main theorem2}]
Assume for simplicity that $\overline{\Omega} \subset \{ x_3 > 0 \}$ and that $x_0 = 0$. Let $\phi(x) = \log \,|x|$ be an LCW in $\mR^3$. As explained in \cite[Section 3B]{KS}, if one chooses new coordinates 
\[
y_1 = \log\,|x|, \quad y' = \frac{x}{\abs{x}}
\]
where $y' \in S^2_+ = \{ x \in S^2 \,;\, x_3 > 0 \}$, then the LCW becomes $\phi(y) = y_1$ and the Euclidean metric becomes 
\[
g(y) = c(y) \left( \begin{array}{cc} 1 & 0 \\ 0 & g_0(y') \end{array} \right)
\]
where $g_0$ is the standard metric on the sphere $S^2$. Since manifolds $\{ x \in S^2_+ \,;\, x_3 \geq c > 0 \}$ are simple, we have reduced matters to Theorem \ref{main theorem}. This proves Theorem \ref{main theorem2}.
\end{proof}

\section{Appendix}

For a compact oriented smooth manifold $M$ with smooth boundary, define 
\[
H_{\Delta} = \{u \in L^2(M, \Lambda M) \,|\, \Lap u \in L^2(M, \Lambda M)\},
\]
where $\Lap$ is the Hodge Laplacian. We let
\[
\|u\|^2_{H_{\Delta}} = \|u\|^2_{L^2} + \|\Lap u\|^2_{L^2}.  
\]
In this appendix, we will prove the following trace theorem for $H_{\Delta}$. 

\begin{lemma}\label{Trace}
The relative and absolute trace maps $u \mapsto (tu, t\delta u)$ and $u \mapsto (t*u, t\delta *u)$ initially defined on the space $C^{\infty}(M, \Lambda M)$ have extensions as bounded linear maps from $H_{\Delta}$ to $H^{-\half}(\partial M, \Lambda \partial M) \times H^{-3/2}(\partial M, \Lambda \partial M)$. The set $C^{\infty}(M, \Lambda M)$ is dense in $H_{\Delta}$ and thus these extensions are unique. The trace maps $u \mapsto (tu, td u)$ and $u \mapsto (t*u, td *u)$ have extensions to the same spaces, and one has for any $u \in H_{\Delta}$ 
\begin{gather}
t(du) = d_{\partial M}(tu), \label{tdu_identity} \\
t(d*u) = d_{\partial M}(t*u). \label{tdstaru_identity} 
\end{gather}
\end{lemma}

Lemma \ref{lemma_hdelta_properties} follows immediately from this result. To begin the proof of Lemma \ref{Trace}, we will need the following two lemmas. Here $\nu$ is meant to signify the normal vector field on $\partial M$, extended into $M$ in a standard way.

\begin{lemma} \label{lemma_poisson_solvability}
Let $j \geq 0$. For any $v \in H^j(M, \Lambda M)$ and $f, h \in H^{j+3/2}(\partial M, \Lambda \partial M)$, there is a unique $u \in H^{j+2}(M, \Lambda M)$ that solves 
\[
\Delta u = v \text{ in } M, \quad tu = f, \ \ t*u = h.
\]
One has $\norm{u}_{H^{j+2}} \leq C(\norm{v}_{H^j} + \norm{f}_{H^{j+3/2}} + \norm{h}_{H^{j+3/2}})$.
\end{lemma}
\begin{proof}
This is \cite[Theorem 3.4.10]{gunther}.
\end{proof}

\begin{lemma}\label{LemmaofRequirement}
For any $(f,h) \in H^{\frac{3}{2}}(\partial M, \Lambda \partial M) \times H^{\frac{1}{2}}(\partial M, \Lambda \partial M)$, there is $v \in H^2(M, \Lambda M)$ such that 
\[
\|v\|_{H^2} \lesssim \|f\|_{H^{3/2}} + \|h\|_{H^{1/2}}
\]
and $ti_{\nu}v = f$, $ti_{\nu}*v = ti_{\nu} d*v = 0$, and $ti_{\nu}d v = h$. 
\end{lemma}
\begin{proof}
Let first $f=0$. Any boundary point has a neighborhood $U \subset M$ where boundary normal coordinates $(x_1, \ldots, x_n)$ exist, so that the boundary is given by $\{ x_n = 0 \}$ and $x_n$ coincides with the direction of $\nu$. We first define $v$ in $U$ using these coordinates.  Then $v$ is represented as a sum of objects of the form
\[
v_{i_1, \ldots, i_k} dx_{i_1} \wedge \ldots \wedge dx_{i_k}
\]
for $1 \leq i_k \leq n$, where $v_{i_1, \ldots, i_k}$ are functions.  Moreover, $h$ can be represented as a sum of objects of the form
\[
h_{i_1, \ldots, i_k} dx_{i_1} \wedge \ldots \wedge dx_{i_k},
\]
for $1 \leq i_k < n$.

We will define $v_{i_1, \ldots, i_k}$ to be identically zero if any $i_k = n$, and choose the remaining $v_{i_1, \ldots, i_k}$ to be functions such that 
\begin{eqnarray*}
v_{i_1, \ldots, i_k} &=& 0 \mbox{ and } \\
\partial_{\nu}v_{i_1, \ldots, i_k} &=& h_{i_1, \ldots, i_k} \\
\end{eqnarray*}
on $U \cap \partial M$. Observe that when chosen this way all components of the form $v_{i_1, \ldots, i_k} dx_{i_1} \wedge \ldots \wedge dx_{i_k} $ vanish on the boundary and  $i_\nu v_{i_1, \ldots, i_k} dx_{i_1} \wedge \ldots \wedge dx_{i_k} $ vanish in a neighborhood of the boundary. We can arrange for $\|v_{i_1, \ldots, i_k}\|_{H^2(U)} \lesssim \|h_{i_1, \ldots, i_k}\|_{H^{1/2}(U \cap \partial M)}$, as an inequality for functions. Using a suitable partition of unity we can define $v$ near $\partial M$ and extend smoothly as an element $v \in H^2(M, \Lambda M)$, and so we will obtain the inequality
\[
\|v\|_{H^2(M, \Lambda M)} \lesssim \|h\|_{H^{\half}(\partial M, \Lambda \partial M)}
\]
for $v$ and $h$ as forms.  Note that our definition is made to guarantee that $v|_{\partial M}= 0$, $t\grad_{\nu} v = h$, and $i_{\nu} v = 0$ in a neighbhorhood of the boundary of $M$.  Now by Lemma 3.4 of ~\cite{CST}, 
\[
ti_{\nu} dv = t\grad_{\nu} v_{\|} + Stv_{\|} -d'ti_{\nu}v
\]
where $S$ is a bounded endomorphism on $\Lambda M$ and $d' $ denotes the exterior derivative on $\partial M$. The last two terms vanish because $v$ is zero on $\partial M$.

Then we are left with 
\[
ti_{\nu} dv = t \grad_{\nu} v_{\|} = t \grad_{\nu} v = h.
\]
 Moreover, $ti_{\nu} v = ti_{\nu} *v  = 0$ since $v$ is zero on the boundary of $M$, and $ti_{\nu} d *v = 0$ since 
\[
ti_{\nu} d*v = t\grad_{\nu} (*v)_{\|} + St(*v)_{\|} -d'ti_{\nu}*v.
\]
and $(*v)_{\|} = *v_{\perp} = 0 $ in a neighbourhood of the boundary.  This proves the lemma when $f=0$. The next step is to consider the case where $f$ is nonzero but $h=0$. An argument in boundary normal coordinates similar to the one above yields a form $v$ with $ti_{\nu}v = f$, $ti_{\nu}*v = ti_{\nu} d*v = ti_{\nu}d v = 0$, and $\norm{v}_{H^2} \lesssim \norm{f}_{H^{3/2}}$. Combining these two cases proves the lemma.  
\end{proof}
\begin{proof}[Proof of Lemma \ref{Trace}]
Let $u,v \in H^2(M, \Lambda M)$.  Then using the integration by parts formulas \eqref{iparts1} and \eqref{iparts2} we get the identity \begin{equation}\label{HodgeIbyP}
(u | \Lap v)_M - (\Lap u| v)_M = (u | i_{\nu}dv - \nu \wedge \delta v)_{\partial M} + ( \nu \wedge \delta u - i_{\nu}du | v)_{\partial M}.
\end{equation}
We can rewrite the boundary terms in terms of the relative and absolute boundary values of $u$, by using the following formulas valid for any $1$-form $\xi$ and any $k$-form $\eta$ (see for example ~\cite{CST}):
\begin{gather*}
i_{\xi} \eta = (-1)^{n(k-1)} *\xi \wedge * \eta, \quad ** \eta = (-1)^{k(n-k)} \eta, \quad \delta \eta = (-1)^{(k-1)(n-k)+1} *d* \eta, \\
\langle *u, *v \rangle = \langle u, v \rangle, \quad \langle \xi \wedge u, v \rangle = \langle u, i_{\xi} v \rangle, \quad \langle i_{\nu} u, v \rangle|_{\partial M} = \langle t i_{\nu} u, t v \rangle
\end{gather*}
After a computation (first assuming that $u$ and $v$ are $k$-forms and then summing over $k$), \eqref{HodgeIbyP} can be written as  
\begin{multline}
(u | \Lap v)_M - (\Lap u| v)_M = (tu| t i_{\nu} d v)_{\partial M}  +(t*u| ti_{\nu}d*v)_{\partial M} \\
 + (t\delta u| ti_{\nu}v)_{\partial M} + (t \delta * u|ti_{\nu} * v)_{\partial M}. \label{HodgeIbyPt}
\end{multline}

Now if $u \in H_{\Delta}$, we wish to define the relative boundary values of $u$ by \eqref{HodgeIbyPt}.  Given any $(f,h) \in H^{\frac{3}{2}}(\partial M, \Lambda \partial M) \times H^{\frac{1}{2}}(\partial M, \Lambda \partial M)$, we use Lemma \ref{LemmaofRequirement} to choose $v = v_{f,h}$ such that $ti_{\nu} v = f$, $ti_{\nu} d v = h$, $ti_{\nu}*v= ti_{\nu} d * v = 0$ and $\|v\|_{H^2} \lesssim \|f\|_{H^{3/2}} + \|h\|_{H^{1/2}}$.  By \eqref{HodgeIbyPt} we have for any $u \in H^2(M, \Lambda M)$ 
\begin{align*}
(tu| h)_{\partial M} &= (u | \Lap v_{0,h})_M - (\Lap u| v_{0,h})_M, \\
(t\delta u |f)_{\partial M} &= (u | \Lap v_{f,0})_M - (\Lap u| v_{f,0})_M.
\end{align*}
Since
\[
\|tu\|_{H^{-\half}} = \sup_{\|h\|_{H^{\half}} = 1} |(tu| h)_{\partial M}|, \qquad \|t\delta u\|_{H^{-3/2}} = \sup_{\|f\|_{H^{3/2}} = 1} |(t\delta u| f)_{\partial M}|,
\]
it follows that 
\[
\|tu\|_{H^{-\half}} + \|t\delta u\|_{H^{-3/2}} \lesssim \|u\|_{H_{\Delta}}, \qquad u \in H^2(M, \Lambda M).
\]
Thus for any $u \in H_{\Delta}$, we may use the above formulas to define $(tu, t\delta u)$ as an element of $H^{-1/2}(\partial M, \Lambda(\partial M)) \times H^{-3/2}(\partial M, \Lambda(\partial M))$. Applying this argument to $*u$ shows that $(t*u, t\delta *u)$ are well defined for $u \in H_{\Delta}$ in a similar way. The identity \eqref{HodgeIbyPt} then remains true for any $u \in H_{\Delta}$ and any $w \in H^2(M, \Lambda M)$.

We now claim that $C^{\infty}(M,\Lambda M)$ is dense in $H_{\Delta}$. Given this fact, it follows that the maps $u \mapsto (tu, t\delta u)$ and $u \mapsto (t*u, t\delta *u)$ are bounded from $H_{\Delta}$ to $H^{-1/2} \times H^{-3/2}$ and are uniquely defined by the corresponding maps acting on $C^{\infty}$. Finally, \eqref{tdu_identity} and \eqref{tdstaru_identity} are true for $u \in C^{\infty}(M, \Lambda M)$ by the fact that $d$ commutes with the pull-back to the boundary $t$ 
and thus they extend to $u \in H_{\Delta}$ by density.

It remains to show that $C^{\infty}(M,\Lambda M)$ is dense in $H_{\Delta}$. Suppose that $u \in H_{\Delta}$. Then $\Lap u \in L^2(M, \Lambda M)$ and $tu, t*u \in H^{-1/2}(\partial M, \Lambda \partial M)$, so for any $\e > 0$ there exist $v \in \Om(M)$ and $f, h \in \Om(\partial M)$ such that 
\begin{eqnarray}
\label{u-v bounds}
\|\Lap u - v\|_{L^2} + \|f - tu\|_{H^{-\half}} + \|h - t* u\|_{H^{-1/2}} < \e.
\end{eqnarray}
Now by Lemma \ref{lemma_poisson_solvability} there exists a smooth form $u'$ such that $\Lap u' = v$, with $tu' = f$ and $t* u' = h$. Then if $w \in H^2(M, \Lambda M)$ satisfies $tw = t*w = 0$, which implies $ti_{\nu} w = ti_{\nu}*w = 0$, we get by \eqref{HodgeIbyPt} that 
\begin{equation*}
(u - u'|\Lap w)_M =(\Lap(u - u')|w)_M  + (t(u - u')| t i_{\nu}d w)_{\partial M}  +(t*(u-u')| ti_{\nu}d*w)_{\partial M}.\end{equation*}
Combining this with inequality \eqref{u-v bounds} we obtain
\[ \abs{(u - u'|\Lap w)_M} \leq C\e\|w\|_{H^2}. \]
Now by Lemma \ref{lemma_poisson_solvability} every form $z \in L^2(M, \Lambda M)$ can be written as $\Lap w$ for some $w$ with $tw = t*w = 0$. Moreover, we will have
\[ \|w\|_{H^2} \lesssim \|z\|_{L^2}. \] 
Thus we get that
\[ \|u - u'\|_{L^2} \lesssim \e \]
 and hence
\[ \|u - u'\|_{H_{\Delta}} \lesssim \e. \]
Therefore there is a sequence of smooth forms $u_k$ such that $u_k \rightarrow u$ in $H_{\Delta}$.
\end{proof}

\bibliographystyle{alpha}

\providecommand{\bysame}{\leavevmode\hbox to3em{\hrulefill}\thinspace}
\providecommand{\href}[2]{#2}

\end{document}